\documentclass[11pt]{amsart}

\usepackage[left]{lineno}

\usepackage{amsmath,amssymb,amsthm}
\usepackage[mathscr]{euscript}
\usepackage{color}
\usepackage{hyperref,doi}
\usepackage{epsf}
\usepackage{enumerate}
\usepackage{graphicx}
\usepackage{array}
\usepackage{color}

\usepackage[margin=1in]{geometry}

\newtheorem{theorem}{Theorem}[section]

\newtheorem{lemma}[theorem]{Lemma}

\newtheorem{corollary}[theorem]{Corollary}
\newtheorem{question}[theorem]{Question}
\theoremstyle{definition}
\newtheorem{definition}[theorem]{Definition}
\newtheorem{remark}[theorem]{Remark}

\newcommand{\conv}[1]{\mathrm{conv}\left\{#1\right\}}

\newcommand{\R}{\mathbb{R}}
\newcommand{\Z}{\mathbb{Z}}

\newcommand{\ba}{\mathbf{a}}

\newcommand{\bp}{\mathbf{p}}
\newcommand{\bq}{\mathbf{q}}

\newcommand{\bv}{\mathbf{v}}

\DeclareMathOperator{\ehr}{Ehr}

\newcommand\commentout[1]{}

\makeatletter
\def\@tocline#1#2#3#4#5#6#7{\relax
  \ifnum #1>\c@tocdepth 
  \else
  \par \addpenalty\@secpenalty\addvspace{#2}%
  \begingroup \hyphenpenalty\@M
  \@ifempty{#4}{%
    \@tempdima\csname r@tocindent\number#1\endcsname\relax
  }{%
    \@tempdima#4\relax
  }%
  \parindent\z@ \leftskip#3\relax \advance\leftskip\@tempdima\relax
  \rightskip\@pnumwidth plus4em \parfillskip-\@pnumwidth
  #5\leavevmode\hskip-\@tempdima
  \ifcase #1
  \or\or \hskip 1em \or \hskip 2em \else \hskip 3em \fi%
  #6\nobreak\relax
  \hfill\hbox to\@pnumwidth{\@tocpagenum{#7}}\par
  \nobreak
  \endgroup
  \fi}
\makeatother

\begin{document}



\title{Ehrhart Limits}

\author{Benjamin Braun}
\address{Department of Mathematics\\
  University of Kentucky\\
  Lexington, KY 40506--0027}
\email{benjamin.braun@uky.edu}

\author{McCabe Olsen}
\address{Department of Mathematics\\
  Rose-Hulman Institute of Technology\\
  Terre Haute, IN 47803--3920}
\email{olsen@rose-hulman.edu}




\date{1 December 2022}

\begin{abstract}
  We introduce the definition of an Ehrhart limit, that is, a formal power series with integer coefficients that is the limit in the ring of formal power series of a sequence of Ehrhart $h^*$-polynomials.
  We identify a variety of examples of sequences of polytopes that yield Ehrhart limits, with a focus on reflexive polytopes and simplices.
\end{abstract}

\thanks{BB was partially supported by National Science Foundation award DMS-1953785.}
\maketitle



\section{Introduction}

Given a lattice polytope $P$, the Ehrhart $h^*$-polynomial of $P$, denoted $h^*(P;z)$, is an important invariant with connections to enumeration, triangulations, commutative algebra, and beyond.
In this work, we consider the following question: which formal power series, i.e., $f(z)\in \Z[[z]]$, are the limit of a sequence of Ehrhart $h^*$-polynomials?
More precisely, for which power series $f$ does there exist a sequence of lattice polytopes $\{P_k\}_{k=1}^\infty$ such that
\[
  \lim_{k\to \infty}h^*(P_k;z)=f(z)\, ?
\]

Note that using the usual topology on $\Z[[z]]$, this limit exists if for each $j\geq 0$, the sequence of coefficients of $z^j$ in $h^*(P_k;z)$ is eventually constant (i.e., it \emph{stabilizes}).
Alternatively stated, viewing the set of all Ehrhart $h^\ast$-polynomials in the subset topology of $\Z[[z]]$, we wish to describe the closure of this set with particular emphasis on elements of the closure that are not polynomials.
Thus, we are most interested in the case where $\dim(P_k)<\dim(P_{k+1})$ for all $k$.
If $f(z)$ is an element of this closure, then we say $f(z)$ is an \emph{Ehrhart limit} and that the sequence of polytopes $P_1,P_2,P_3,\ldots$ \emph{converges to an Ehrhart limit}.

We first remark that every $h^*$-polynomial is itself an Ehrhart limit arising from polytopes of increasing dimension, as the operation of taking lattice pyramids preserves the Ehrhart $h^*$-polynomial.
A second remark is that many sequences of $h^*$-polynomials that are standard examples do not converge.
For example, the polynomial $(1+z)^k$ is the $h^*$-polynomial of the crosspolytope, and the binomial coefficients have wonderful properties; however, this sequence of polynomials has a strictly increasing sequence of linear coefficients.
Similarly, the Eulerian polynomials are $h^*$-polynomials for the unit cubes, and this sequence also does not converge.
Thus, what is required for an Ehrhart limit is that we have a sequence of $h^*$-polynomials where the initial segments of coefficients of the earlier polynomials are identical to the initial segments of later polynomials.
This is not a property that many of the ``usual examples'' in Ehrhart theory have, but there are many examples of sequences of polytopes with this property as we will demonstrate.
One notable condition that an Ehrhart limit must have is that any universal inequality on $h^*$-polynomials, such as those given in~\cite{universalehrhart}, must be satisfied.

Our main goal in this work is to introduce the definition of an Ehrhart limit and identify sequences of lattice polytopes whose $h^*$-polynomials converge.
The sequences of polytopes that we investigate fall into two categories.
First, in Section~\ref{sec:reflexivelimits}, we show that reflexive polytopes are a fruitful source of Ehrhart limits.
Some of these limits arise from the free sum operation applied to a reflexive polytope and reflexive simplices of minimal volume.
Others arise through an operation known as reflexive stabilization~\cite{idpweightedprojectivespace}.
For the Ehrhart limits we consider arising from reflexive polytopes, we are able to describe the actual power series that is the limit.
We also prove that products of Ehrhart limits are Ehrhart limits, using the join operation.

Second, in Section~\ref{sec:limits}, we show that many multi-diagonal simplices in Hermite normal form yield Ehrhart limits.
This section is where most of our work takes place, as the proofs that these limits exist involve careful analysis of the fundamental parallelepiped points for these simplices.
One interesting aspect of these proofs is that while we can prove that many sequences of these $h^*$-polynomials converge, it is challenging to explicitly determine the power series that is the limit.
We are able to find an explicit description for one such sequence.

We end the paper in Section~\ref{sec:conclusion} with some further questions.
By providing evidence that these limits exist and arise naturally, we hope to generate interest in Ehrhart limits and the structure of the resulting power series.


\section{Ehrhart Theory Background}\label{sec:background}

A subset $P\subseteq\R^n$ is a $d$-dimensional (convex) \emph{lattice polytope} if it is the convex hull of finitely many points $\bv^{(1)},\ldots,\bv^{(k)}\in\Z^n$ that span a $d$-dimensional affine subspace of $\R^n$.
We say $P$ is a \emph{$d$-simplex} if $P$ is $d$-dimensional and has exactly $d+1$ vertices.
The \emph{Ehrhart polynomial} of $P$ is the lattice point enumerator $i(P;t):=|tP\cap\Z^n|$, where $tP:=\{t\bp:\bp\in P\}$ denotes the $t^{th}$ dilate of the polytope $P$, while the \emph{Ehrhart series} of $P$ is the rational function
\[
  \ehr_P(z) := \sum_{t\geq0}i(P;t)z^t = \frac{h_0^*+h_1^*z+\cdots+h_d^*z^d}{(1-z)^{\dim(P)+1}} \, .
\]
That this is a rational function is a consequence of a theorem due to Ehrhart~\cite{ehrhart}.
It is also known that the coefficients $h^\ast_0,h^\ast_1,\ldots,h^\ast_d$ are all nonnegative integers \cite{StanleyDecompositions}, and the polynomial $h^\ast(P;z):=h_0^*+h_1^*z+\cdots+h_d^*z^d$ is called the \emph{(Ehrhart) $h^\ast$-polynomial} of $P$.

When $\Delta$ is a lattice simplex, there is a wonderful arithmetic interpretation of the $h^*$-polynomial.
If $\Delta=\conv{\bv^{(1)},\ldots,\bv^{(d+1)}}$, the \emph{fundamental parallelepiped} for $\Delta$ is
\[
  \Pi_\Delta := \left\{\sum_{i=1}^{d+1}\lambda_i(1,\bv^{(i)}):0\leq \lambda_i<1 \text{ for all }i \right\}\, .
\]
In this case, it straightforward to show via a tiling argument that
\begin{equation}\label{eq:fppsum}
h^*(\Delta;z)=\sum_{(m_0,m_1,\ldots,m_{d})\in \Z^{d+1}\cap \Pi_\Delta}z^{m_0} \, ,
\end{equation}
i.e., the $h^*$-polynomial is the generating function for the heights of the lattice points in $\Pi_\Delta$.


\section{Ehrhart Limits from Reflexive Polytopes}\label{sec:reflexivelimits}

We begin with two examples of Ehrhart limits arising from reflexive polytopes.
Recall that a lattice polytope $P$ is \emph{reflexive} if $P$ contains the origin in its interior and both $P$ and the dual of $P$ are lattice polytopes.
Reflexive polytopes have the following useful property with regard to the free sum operation.

\begin{theorem}[Braun, \cite{BraunEhrhartFormulaReflexivePolytopes}]\label{thm:freesum}
  Given a reflexive polytope $P$ and a lattice polytope $Q$ whose interior contains the origin, the $h^*$-polynomial of the free sum $P\oplus Q$ is given by $h^*(P;z)h^*(Q;z)$.
\end{theorem}

Recall that the simplex $S_d:=\conv{e_1,e_2,\ldots,e_d,-\sum_{i=1}^de_i}$ is reflexive and $h^*(S_d;z)=\sum_{i=0}^dz^i$.
Using $S_d$, we can prove the following theorem, establishing our first collection of Ehrhart limits.

\begin{theorem}\label{thm:freesumlimit}
  If $Q$ is a lattice polytope with the origin in its interior, and if $k\in \Z_{\geq 1}$, then
  \[
    \frac{h^*(Q;z)}{(1-z)^k}
  \]
  is an Ehrhart limit.
\end{theorem}

\begin{proof}
  Consider the sequence of polytopes $Q\oplus \left(\oplus_1^k S_d\right)$ for $d\geq 1$.
  By Theorem~\ref{thm:freesum}, their $h^*$-polynomials are
  \[
    h^*\left(Q\oplus \left(\oplus_1^k S_d\right);z\right)=h^*(Q;z)\left(\sum_{i=0}^dz^i\right)^k \, .
  \]
  As $d\to \infty$, the limit of these $h^*$-polynomials is the series $\frac{h^*(Q;z)}{(1-z)^k}$, as desired.  
\end{proof}

Theorem~\ref{thm:freesumlimit} can be extended using the join operation, where the \emph{join} of two lattice polytopes $P\subset \R^{d_P}$ and $Q\subset \R^{d_Q}$ is
\[
  P\star Q = \conv{P\times\{\mathbf{0}_{d_Q}\}\times \{0\} \cup \{\mathbf{0}_{d_P}\}\times Q \times \{1\}} \subset \R^{d_P+d_Q+1}\, .
\]
It is a standard exercise to prove that $h^*(P\star Q;z)=h^*(P;z)h^*(Q;z)$ for any lattice polytopes $P$ and $Q$.
The following result is an immediate consequence.

\begin{theorem}\label{thm:multiplication}
  If $\{P_i:i\in \Z_{\geq 0}\}$ converges to an Ehrhart limit $f(z)$ and $\{Q_i:i\in \Z_{\geq 0}\}$ converges to an Ehrhart limit $g(z)$, then $\{P_i\star Q_i:i\in \Z_{\geq 0}\}$ converges to the Ehrhart limit $f(z)g(z)$.
Thus, the set of Ehrhart limits is closed under multiplication.
  \end{theorem}

Our second example does not arise from free sums, but it is the Ehrhart limit of a sequence of reflexive simplices.
Given a vector $\bq\in \Z_{\geq 1}^d$, define the $d$-dimensional simplex
\[
  \Delta_{(1,\bq)}:=\conv{e_1,e_2,\ldots,e_d,-\bq} \, .
\]
\begin{theorem}[Braun, Davis, Hanely, Lane, Solus, \cite{idpweightedprojectivespace}]\label{thm:almostidpunimodal}
  For $n\geq 1$, define
  \[
    \bq(n):=(\underbrace{1,1,\ldots,1}_{2n-1 \text{ times}},3n,10n,15n) \, .
  \]
  For $n\geq 2$, $\Delta_{(1,\bq(n))}$ is reflexive and
  \[
    h^*(\Delta_{(1,\bq(n))};z)=(1+z^2+z^4+z^6+\cdots+z^{2n-2})\cdot (1+7z+14z^2+7z^3+z^4) \, ,
  \]
  which has coefficient vector $(1, 7, 15, 14, 16, 14,16,14,\ldots, 16, 14, 16, 14, 15, 7, 1)$.
\end{theorem}

Letting $n\to \infty$ in the preceding theorem, we obtain the following corollary.

\begin{corollary}\label{cor:delta1qlimit}
  The series $1+7z+15z^2+14z^3+16z^4+14z^5+16z^6+14z^7+16z^8+\cdots$ is an Ehrhart limit.
\end{corollary}

\section{Ehrhart Limits From Simplices}\label{sec:limits}

In this section, we focus our attention on sequences of lattice simplices with specific Hermite normal forms.
Recall that a simplex is in Hermite normal form if it has the origin as a vertex and the other vertices are the columns of a matrix that is upper triangular, where the entries in the $j$-th column are non-negative integers strictly less than the $j$-th diagonal element.
Every lattice simplex is unimodularly equivalent to a unique lattice simplex in Hermite normal form~\cite{schrijverbook}.
One of the interesting aspects of the sequences we study in this section is that while we are able to prove that their $h^*$-polynomials converge, it seems difficult to determine an explicit form for their limit.

In the first subsection, we will prove convergence for special sequences of Hermite normal form simplices, for one of which we can explicitly describe the limit.
In the second subsection, we consider a more general collection of sequences of Hermite normal form simplices and prove that those sequences yield Ehrhart limits.
The results in this section were motivated in part by experiments conducted using SageMath~\cite{sage}.

\subsection{Bidiagonal Matrices}\label{sec:mcase}

In this subsection, we focus on the following simplices.

\begin{definition}\label{def:Pmd}
For $m\in \Z$ with $m\geq 2$, we denote by $P_{m,d}$ the $(d-1)$-dimensional simplex with vertices given by the columns of the $(d-1)\times d$ matrix
\[
\left[
    \begin{array}{cccccccccc}
      0 & 1 & 1 & 0 & 0 & 0 & 0 & \cdots & 0 \\
      0 & 0 & m & 1 & 0 & 0 & 0 & \cdots & 0 \\
      0 & 0 & 0 & m & 1 & 0 & 0 & \cdots & 0 \\
      0 & 0 & 0 & 0 & m & 1 & 0 & \cdots & 0 \\
      0 & 0 & 0 & 0 & 0 & m & 1 & \cdots & 0 \\
      \vdots & \vdots & \vdots &  \vdots & \vdots & \vdots & \ddots & \ddots &\vdots  \\
      0 & 0 & 0 & 0 & 0 & 0 & 0 & \ddots & 1 \\
      0 & 0 & 0 & 0 & 0 & 0 & 0 & \cdots & m \\
    \end{array}
  \right] \, .
\]
\end{definition}

Our main results in this subsection are the following.

\begin{theorem}\label{thm:mehrhartlimit}
  For fixed $m\geq 2$, the polytopes $P_{m,d}$ converge to an Ehrhart limit as $d\to \infty$.
\end{theorem}

For $m=2$, we can find an explicit description of the limit.

\begin{definition}
  For positive integers $k$ and $h$, define the function
  \[
    f(k,h):=2\left[\binom{k-2}{3(h-1)-k}+\binom{k-2}{3(h-1)-k-1}+\binom{k-2}{3(h-1)-k-2}  \right] \, .
  \]
\end{definition}

\begin{theorem}\label{thm:m2limit}
  Denoting by $H(z)=\sum_{i=0}^\infty h_iz^i$ the Ehrhart limit of $P_{2,d}$, we have
  \[
    h_j=\sum_{k=\left\lceil \frac{3}{2}(h-1)\right\rceil}^{3h-3}f(k,h) \, .
  \]
  Thus, $H(z)= 1 +z+ 4z^2+ 20z^3 +84z^4+ 356z^5 +1508z^6 +6388z^7+ 27060z^8+ 114628z^9+ 485572z^{10}+\cdots$.
\end{theorem}

\begin{proof}[Proof of Theorem~\ref{thm:m2limit}]
Theorem~\ref{thm:m2limit} follows directly from~\eqref{eq:fppsum} and Theorem~\ref{thm:fkhenumeration} below.
\end{proof}

\begin{remark}
  Observe that for $i\geq 4$, the coefficients in this power series satisfy the linear recursion $h_i=4h_{i-1}+h_{i-2}$.
\end{remark}

Our proofs of Theorem~\ref{thm:mehrhartlimit} and~\ref{thm:m2limit} rely on a careful analysis of the lattice points in the fundamental parallelepiped of $P_{m,d}$.

\begin{theorem}\label{thm:mdfpp}
  Every lattice point in the fundamental parallelepiped of $P_{m,d}$ corresponds uniquely to a value, denoted $\lambda_2$, of the following form:
  \[
    \lambda_2 := \frac{\ell_2m^{k-1}+mi_2+j_2}{m^{k}} 
  \]
  where $0<k\leq d-2$, $0\leq \ell_2<m$, $0\leq i_2<m^{k-2}$, and $0<j_2<m$.
\end{theorem}

\begin{proof}
  Suppose that
  \begin{equation}\label{eq:fppmatrix}
    \left[
      \begin{array}{cccccccccc}
        1 & 1 & 1 & 1 & 1 & 1 & 1 & \cdots & 1 \\
        0 & 1 & 1 & 0 & 0 & 0 & 0 & \cdots & 0 \\
        0 & 0 & m & 1 & 0 & 0 & 0 & \cdots & 0 \\
        0 & 0 & 0 & m & 1 & 0 & 0 & \cdots & 0 \\
        0 & 0 & 0 & 0 & m & 1 & 0 & \cdots & 0 \\
        0 & 0 & 0 & 0 & 0 & m & 1 & \cdots & 0 \\
        \vdots & \vdots & \vdots &  \vdots & \vdots & \vdots & \ddots & \ddots &\vdots \\
        0 & 0 & 0 & 0 & 0 & 0 & 0 & \ddots & 1 \\
        0 & 0 & 0 & 0 & 0 & 0 & 0 & \cdots & m \\ 
      \end{array}
    \right]
    \left[
      \begin{array}{c}
        \lambda_0 \\
        \lambda_1 \\
        \lambda_2 \\
        \lambda_3 \\
        \lambda_4 \\
        \vdots  \\
        \lambda_{d-2} \\
        \lambda_{d-1} \\
      \end{array}
    \right]
    = \bp
    \in
    \Z^d
  \end{equation}
  is a lattice point in the fundamental parallelepiped for $P_{m,d}$, where each $0\leq\lambda_i<1$.
  Since the determinant of this matrix is $m^{d-2}$, there are $m^{d-2}$ lattice points in the fundamental parallelepiped.
  We will construct these lattice points explicitly, hence proving the result.

  Fix $0<k\leq d-2$, and set
  \[
    \lambda_2 := \frac{\ell_2m^{k-1}+mi_2+j_2}{m^{k}} 
  \]
  where $0\leq \ell_2<m$, $0\leq i_2<m^{k-2}$, and $0<j_2<m$.
  Note that there are $m^k-m^{k-1}$ such values of $\lambda_2$.
  From the third line of the matrix multiplication producing $\bp$, we see that
  \[
    m \frac{\ell_2m^{k-1}+mi_2+j_2}{m^{k}}+ \lambda_3 \in \Z
  \]
  which implies
  \[
    \frac{mi_2+j_2}{m^{k-1}}+ \lambda_3 \in \Z
  \]
  and since by definition $0<mi_2+j_2<m^{k-1}$, we have
  \[
    \lambda_3 = 1 - \frac{mi_2+j_2}{m^{k-1}} = \frac{m^{k-1}-(mi_2+j_2)}{m^{k-1}} \, .
  \]

  We now rewrite the numerator in the fraction for $\lambda_3$; there exists $0\leq \ell_3<m$, $0\leq i_3<m^{k-3}$, and $0<j_3<m$ such that
  \[
    m^{k-1}-(mi_2+j_2) = \ell_3m^{k-2}+mi_3+j_3  \, ,
  \]
  and we have
  \[
    \lambda_3 = \frac{\ell_3m^{k-2}+mi_3+j_3}{m^{k-1}}\, .
  \]
  We can now use the fourth line in the matrix equation above, following the same process as we used for $\lambda_2$, to find
  \[
    \lambda_4 = \frac{m^{k-2}-(mi_3+j_3)}{m^{k-2}} = \frac{\ell_4m^{k-3}+mi_4+j_4}{m^{k-2}}
  \]
  with $0\leq \ell_4<m$, $0\leq i_4<m^{k-4}$, and $0<j_4<m$.
  Continuing in this fashion, we construct a unique sequence $\{\lambda_t:t\geq 2\}$ where
  \[
    \lambda_t = \frac{m^{k-t+2}-(mi_{t-1}+j_{t-1})}{m^{k-t+2}} = \frac{\ell_tm^{k-t+1}+mi_t+j_t}{m^{k-t+2}}
  \]
  for certain $0\leq \ell_t<m$, $0\leq i_t<m^{k-t}$, and $0<j_t<m$ that are determined uniquely by the process.
  Note that when $t=k-1$ we have
  \[
    \lambda_{k-1}= \frac{j_t}{m}
  \]
  where $0<j_t<m$, and thus for $t\geq k$ we have $\lambda_t=0$.
  This sequence also uniquely determines $\lambda_1$ and $\lambda_0$.

  We must have $k-3<d$ which implies that $1\leq k\leq d-2$.
  Further, each choice of $k$ and an initial value of $\lambda_2$ yields a unique lattice point.
  Since there are $m^k-m^{k-1}$ such points where the denominator of $\lambda_2$ is $m^k$, that means that this yields a total of $\sum_{k=1}^{d-2}(m^k-m^{k-1})=m^{d-2}-1$ lattice points.
  Setting all $\lambda_i$ equal to $0$ yields the $m^{d-2}$-th lattice point in the fundamental parallelepiped.
\end{proof}

\begin{proof}[Proof of Theorem~\ref{thm:mehrhartlimit}]
  Using the notation from the proof of Theorem~\ref{thm:mdfpp}, observe that for any fixed $k$, if $\lambda_2$ has denominator $m^k$ and $k\leq d-2$, then for any such $d$ we obtain the same sequence of coefficients $\lambda_0,\lambda_1,\lambda_2,\ldots$, and the height of the corresponding lattice point in the fundamental parallelepiped is the sum of these coefficients.
  Thus, if we can show that the height of such a fundamental parallelepiped point is bounded below by an increasing function of $k$, it will follow that the coefficients of the $h^*$-polynomial of $P_{m,d}$ stabilize as $d\to \infty$.

  Recall that $0\leq i_t<m^{k-t}$ and $0<j_t<m$, thus $0<mi_t+j_t<m^{k-t+1}$.
  Note that for any $t\geq 2$ such that $\lambda_t\neq 0$ and $\lambda_{t+1}\neq 0$, we have
  \begin{align*}
    \lambda_t+\lambda_{t+1}&= \frac{\ell_tm^{k-t+1}+mi_t+j_t}{m^{k-t+2}} + \frac{\ell_{t+1}m^{k-t}+mi_{t+1}+j_{t+1}}{m^{k-t+1}}\\
                           & = \frac{\ell_tm^{k-t+1}+mi_t+j_t}{m^{k-t+2}} + \frac{m^{k-t+1}-(mi_t+j_t)}{m^{k-t+1}} \\
                           & = 1 + \frac{\ell_tm^{k-t+1}-(m-1)(mi_t+j_t)}{m^{k-t+2}}  \\
                           & > 1 + \frac{\ell_tm^{k-t+1}-(m-1)m^{k-t+1}}{m^{k-t+2}}  \\
                           & = 1 + \frac{\ell_t-(m-1)}{m}  \\
                           & > 1 - \frac{m-1}{m}  \\
                           & = \frac{1}{m} \, .
  \end{align*}
  Thus, for any point defined by coefficients where $\lambda_2$ has denominator $m^k$, it follows that the height of the point is at least $\lfloor k/2 \rfloor\frac{1}{m}$.
  Hence, for fixed degree $r$, as $d\to \infty$ the coefficient of $z^r$ in $h^*(P_{m,d};z)$ is eventually constant.
\end{proof}

Considering the special case of $m=2$, Theorem~\ref{thm:mehrhartlimit} implies that when $h$ is fixed, every lattice point at height $h$ in the fundamental parallelepiped of $P_{2,d}$ with $\lambda_2=b/2^k$ for an odd value of $b$ satisfies $\lfloor k/2 \rfloor\frac{1}{2}\leq h$.
This implies that as $d$ grows, the number of fundamental parallelepiped points for $P_{2,d}$ at height $h$ stabilizes and can be found among those points having a bounded denominator of $\lambda_2$.
To prove Theorem~\ref{thm:m2limit}, we establish an explicit description of the Ehrhart limit of $P_{2,d}$ by enumerating the set of possible fundamental parallelepiped points for arbitrary $P_{2,d}$ having a given height and a given denominator of $\lambda_2$.
Note that
\[
  \left\lceil \frac{3}{2}(h-1)\right\rceil \leq k \leq 3h-3
\]
if and only if
\[
  \left\lceil\frac{k+3}{3}\right\rceil \leq h \leq \left\lfloor \frac{2k+3}{3}\right\rfloor \, ,
\]
establishing a connection between Theorem~\ref{thm:m2limit} and the following lemma.

\begin{lemma}\label{lem:sumfpowerof2}
  For a positive integer $k$, we have
  \[
\sum_{h=\left\lceil\frac{k+3}{3}\right\rceil}^{\left\lfloor \frac{2k+3}{3}\right\rfloor}f(k,h)=2^{k-1} \, .
    \]
  \end{lemma}

  \begin{proof}
    Write $k=3q+r$ with $0\leq r <3$, and proceed by cases for $r=0,1,2$.
    When $r=0$, we have that $k=3q$, implying that the index of summation ranges over $q+1\leq h\leq 2q+1$.
    Evaluating $f(k,h)$ at these values of $h$ shows that the right-hand side above is equal to
    \[
2\sum_{i=0}^{3q-2}\binom{3q-2}{i} = 2\cdot 2^{3q-2}=2^{k-1}
\]
and thus this case is complete.
The remaining two cases follow similarly.
\end{proof}

Thus, the summation in Lemma~\ref{lem:sumfpowerof2} ranges over pairs $(k,h)$ for fixed $k$ while the summation in Theorem~\ref{thm:m2limit} ranges over pairs $(k,h)$ with fixed $h$, but the set of possible pairs appearing in any of these summations is the same.
Our next result directly implies Theorem~\ref{thm:m2limit}.

\begin{theorem}\label{thm:fkhenumeration}
  Fix $k$ and $h$ such that $\left\lceil \frac{3}{2}(h-1)\right\rceil \leq k \leq 3h-3$.
  For sufficiently large $d$, the number of fundamental parallelepiped points for $P_{2,d}$ having height $h$ and $\lambda_2=b/2^k$ is equal to $f(k,h)$.
\end{theorem}

\begin{proof}
  Given that
  \[
    f(k,h)=2\left[\binom{k-2}{3(h-1)-k}+\binom{k-2}{3(h-1)-k-1}+\binom{k-2}{3(h-1)-k-2}  \right]
  \]
  and
  \[
\sum_{h=\left\lceil\frac{k+3}{3}\right\rceil}^{\left\lfloor \frac{2k+3}{3}\right\rfloor}f(k,h)=2^{k-1} \, ,
\]
our procedure is to find a bijection from pairs of fundamental parallelepiped points having height $h$ and $\lambda_2=b/2^k$ to $(3(h-1)-k-x)$-subsets of $\{0,1,\ldots,k-3\}$ where $x=0,1,2$.
Following the notation of the proof of Theorem~\ref{thm:m2limit}, let $(\lambda_0,\lambda_1,\ldots,\lambda_{k+1})$ denote the coefficient vector of a fundamental parallelepiped point.
Note that it is always the case that $\lambda_{k+1}=1/2$, and note that $\lambda_0$ is determined by the other values.
Further, since $\lambda_2+\lambda_1=1$, we observe that our $\lambda$-vectors arise in pairs having the same height: if $\lambda_2<1/2$, we have a pair
\[
(\lambda_0,\lambda_1,\lambda_2,\ldots,\lambda_{k+1}),(\lambda_0,\lambda_1-1/2,\lambda_2+1/2,\ldots,\lambda_{k+1}) \, .
\]
Thus, we can restrict our attention to only those $\lambda$'s having $\lambda_2<1/2$.

Consider the sub-vector $\lambda^s=(\lambda_3,\lambda_4,\ldots,\lambda_k)$, which has $k-2$ entries.
Given a set $B\subseteq \{0,1,2,\ldots,k-3\}$ of size $(3(h-1)-k-x)$ where $x=0,1,2$, we claim that there is a unique $\lambda$ of height $h$ where $\lambda^s$ has exactly $(3(h-1)-k-x)$ entries greater than $1/2$.
From this, Lemma~\ref{lem:sumfpowerof2} implies that these are all possible fundamental parallelepiped points where the denominator of $\lambda_2$ is $2^k$, and our proof will be complete.

If there are no elements of $\lambda^s$ greater than $1/2$, then it is a straightforward exercise, starting from $\lambda_{k+1}=1/2$, to deduce from~\eqref{eq:fppmatrix} that for $j=-1,0,\ldots,k-2$, we have
\[
\lambda_{k-j}=\frac{1}{2^{j+2}}\sum_{i=0}^{j+1}(-1)^i2^{j+1-i} \, .
\]
(An interesting aside: the resulting numerators of these fractions form the \emph{Jacobsthal sequence}~\cite{oeisjacobsthal}.)
Thus, applying various manipulations of finite geometric series yields that
\begin{equation}\label{eq:mu}
\lambda_1+\lambda_2+\cdots+\lambda_{k+1}=1+\frac{k+1}{3}+\frac{1}{9}\left[1-(-1/2)^{k-1}\right] \, .
  \end{equation}
  If there exists some $t\in\{0,1,\ldots,k-3\}$ such that $\lambda_{k-t}>1/2$, then the matrix equation~\eqref{eq:fppmatrix} forces us to add to~\eqref{eq:mu} the value
  \[
2^{t+1}\sum_{i=0}^{k-t-3}(-1/2)^i = \frac{1}{3}\left[ 1-(-1/2)^{k-t-2}\right] \, ,
\]
since the addition of $1/2$ to $\lambda_{k-t}$ propagates through the subsequent the values of $\lambda$.
Thus, if $\lambda^s$ has $\alpha$ entries greater than $1/2$, i.e., if $\lambda_{k-t_r}>1/2$ for $r=1,2,\ldots,\alpha$, then the height of the fundamental parallelepiped point defined by $\lambda$ is the ceiling of
\begin{equation}\label{eq:height}
  1+\frac{k+1}{3}+\frac{1}{9}\left[1-(-1/2)^{k-1}\right] + \sum_{r=1}^\alpha \frac{1}{3}\left[ 1-(-1/2)^{k-t_r-2}\right] \, .
\end{equation}
If $\alpha=3(h-1)-k-x$ for $x=0,1,2$, then~\eqref{eq:height} becomes
\begin{equation}\label{eq:finalheight}
h-\frac{x+1}{3}+\frac{1}{9}\left[1-(-1/2)^{k-1}\right]-\frac{1}{3}\sum_{r=1}^\alpha(-1/2)^{k-t_r-2} \, .
\end{equation}
To verify that this point has height $h$, we must show that
\begin{equation}\label{eq:bounded}
\frac{x+1}{3}-\frac{1}{9}\left[1-(-1/2)^{k-1}\right]+\frac{1}{3}\sum_{r=1}^\alpha(-1/2)^{k-t_r-2}
  \end{equation}
  is contained in the interval $[0,1)$.

  To prove this, first observe that the smallest value that $k-t_r-2$ can have is $1$.
  Thus, by summing only the largest-magnitude odd- and even-power terms that might arise in the sum, respectively, we obtain
  \begin{equation}\label{eq:bounds}
    \frac{-1}{6}\left(1-(1/4)^\alpha\right)\leq \sum_{r=1}^\alpha(-1/2)^{k-t_r-2} \leq \frac{1}{3}\left(1-(1/4)^\alpha\right) \, .
  \end{equation}
  We next consider three cases.
  If $x=0$ or $x=1$, then applying~\eqref{eq:bounds} to~\eqref{eq:bounded}, it follows immediately that the value is contained in $[0,1)$.

  If $x=2$, the lower bound in~\eqref{eq:bounds} applied to~\eqref{eq:bounded} immediately yields the desired lower bound.
  However, the upper bound is slightly more delicate.
  When $k$ is even, then regardless of the value of $\alpha$, we obtain by applying the upper bound in~\eqref{eq:bounds} to~\eqref{eq:bounded} the inequality
  \[
1-\frac{1}{9}(1-(-1/2))^{k-1}+\frac{1}{9}(1-(1/4)^\alpha)=1-(1/2)^{k-1}-(1/4)^\alpha<1 \, ,
\]
and thus for even $k$ the upper bound is immediate.

For the case where $k$ is odd, assume $k=2i+1$ for some $i\geq 0$, and set $\alpha=3(h-1)-k-2$.
Note that since $\{t_1,\ldots,t_\alpha\}\subset \{0,1,\ldots,k-3\}$, to obtain the largest possible value of~\eqref{eq:bounded} we can assume all the summands in $\displaystyle \sum_{r=1}^\alpha(-1/2)^{k-t_r-2}$ are positive and thus $\alpha\leq \frac{k-2}{2}$.
This implies that $3(h-1)-k-2 \leq \frac{k-2}{2}$.
Setting $k=2i+1$ implies that $h\leq i+11/6$.
Since $h$ is an integer, we thus have $i\leq h+1$.
Thus, applying $x=2$, $k=2i+1$, and $i\leq h+1$ to~\eqref{eq:bounded} yields that
\[
  \frac{x+1}{3}-\frac{1}{9}\left[1-(-1/2)^{k-1}\right]+\frac{1}{3}\sum_{r=1}^\alpha(-1/2)^{k-t_r-2} \leq 1-7(1/4)^i< 1 \, .
\]
Thus, our proof is complete.  
\end{proof}

\subsection{Multidiagonal Matrices}\label{sec:multidiagcase}

The matrices defining the simplices in Theorem~\ref{thm:mehrhartlimit} are closely related to the following more general class.

\begin{definition}\label{def:Pmd}
  Suppose $\ba=(a_1,\ldots,a_s)\in \Z_{\geq 1}$ with $a_1>a_j$ for all $2\leq j\leq s$.
  For $d\geq s$, we define the \emph{$\ba$-multidiagonal Hermite normal form simplex}, denoted by $P(\ba;d)=P(a_1,\ldots,a_s;d)$, to be the $d$-dimensional simplex with vertices given by the columns of the following $d\times (d+1)$ matrix:
  
\[
\left[
    \begin{array}{cccccccccc}
      0 & a_1 & a_2 & \cdots & a_s & 0 & 0 & \cdots & 0 \\
      0 & 0 & a_1 & a_2 & \cdots & a_s & 0 & \cdots & 0 \\
      0 & 0 & 0 & a_1 & a_2 & \cdots & a_s & \cdots & 0 \\
      0 & 0 & 0 & 0 & a_1 & a_2 & \cdots & \ddots & 0 \\
      0 & 0 & 0 & 0 & 0 & a_1 & a_2 & \cdots & a_s \\
      \vdots & \vdots & \vdots &  \vdots & \vdots & \vdots & \ddots & \ddots &\vdots  \\
      0 & 0 & 0 & 0 & 0 & 0 & 0 & \ddots & a_2 \\
      0 & 0 & 0 & 0 & 0 & 0 & 0 & \cdots & a_1 \\
    \end{array}
  \right] \, .
\]
\end{definition}

Our main result of this subsection is the following.

\begin{theorem}\label{thm:multilimit}
  Suppose that $\gcd(a_1,a_2)=1$ and
  \[
\left[
  \begin{array}{cccccccccc}
    1 & 1 & 1 & \cdots & 1 & 1 & 1 & \cdots & 1 \\
      0 & a_1 & a_2 & \cdots & a_s & 0 & 0 & \cdots & 0 \\
      0 & 0 & a_1 & a_2 & \cdots & a_s & 0 & \cdots & 0 \\
      0 & 0 & 0 & a_1 & a_2 & \cdots & a_s & \cdots & 0 \\
      0 & 0 & 0 & 0 & a_1 & a_2 & \cdots & \ddots & 0 \\
      0 & 0 & 0 & 0 & 0 & a_1 & a_2 & \cdots & a_s \\
      \vdots & \vdots & \vdots &  \vdots & \vdots & \vdots & \ddots & \ddots &\vdots  \\
      0 & 0 & 0 & 0 & 0 & 0 & 0 & \ddots & a_2 \\
      0 & 0 & 0 & 0 & 0 & 0 & 0 & \cdots & a_1 \\
    \end{array}
  \right]
  \left[
      \begin{array}{c}
        \lambda_0 \\
        \lambda_1 \\
        \lambda_2 \\
        \lambda_3 \\
        \lambda_4 \\
        \vdots  \\
        \lambda_{d-1} \\
        \lambda_{d} \\
      \end{array}
    \right]
    = \left[
      \begin{array}{c}
        p_0 \\
        p_1 \\
        p_2 \\
        p_3 \\
        p_4 \\
        \vdots  \\
        p_{d-1} \\
        p_{d} \\
      \end{array}
    \right]
    =
    \bp
    \in
    \Z^{d+1}
  \]
  is a point in the fundamental parallelepiped for $P(\ba;d)$.
  If $k$ is the largest index such that $\lambda_k\neq 0$, then
  \[
    p_0\geq \left\lfloor \frac{k}{s}\right\rfloor \frac{1}{a_1}
  \]
  (note that $p_0$ is the height of $\bp$ in $\Pi_{P(\ba;d)}$).
  Thus, as $d\to \infty$, $P(\ba;d)$ converges to an Ehrhart limit.
\end{theorem}

\begin{proof}
  Without loss of generality, assume that $k=d$.
  Suppose that $\lambda_k=\frac{x_k}{a_1}$ for some $0<x_1< a_1$, which must occur since $p_k\in \Z$ and $\lambda_k\neq 0$.
  From the second-to-bottom row of the matrix equation above, it follows that for some $t_{k-1}$,
  \[
    \lambda_{k-1}a_1 + a_2\frac{x_k}{a_1} = \frac{\lambda_{k-1}a_1^2 + a_2x_k}{a_1}=t_{k-1}\in \Z_{\geq 1} \, .
  \]
  Thus, solving for $\lambda_{k-1}$, we have
  \[
\lambda_{k-1}= \frac{t_{k-1}}{a_1}-\frac{a_2x_k}{a_1^2} =: \frac{x_{k-1}}{a_1^2} 
\]
for an integer $x_{k-1}$.
Note that since $\gcd(a_1,a_2)=1$ and $0< x_k<a_1$, we have $a_1 \nmid a_2x_k$.
Thus, $\frac{a_2x_k}{a_1^2}$ does not reduce to a fraction with denominator $a_1$, and hence $\frac{t_{k-1}}{a_1}-\frac{a_2x_k}{a_1^2}$ is a non-zero fraction with denominator $a_1^2$.
Thus, $0<x_{k-1}<a_1^2$ and $a_1\nmid x_{k-1}$. 

Proceeding to the next row of the matrix equation above, and using the same process as above, we find there exists a $t_{k-2}\in \Z_{\geq 1}$ such that
\[
\lambda_{k-2}= \frac{t_{k-2}a_1-a_3x_k}{a_1^2}-\frac{a_2x_{k-1}}{a_1^3} =: \frac{x_{k-2}}{a_1^3} 
\]
for an integer $x_{k-2}$.
Similarly to our previous step, we have $\gcd(a_1,a_2)=1$ and $a_1\nmid x_{k-1}$ implies that $\frac{a_2x_{k-1}}{a_1^3}$ fully reduced has denominator $a_1^3$.
Thus, $\frac{x_{k-2}}{a_1^3}$ is a non-zero fraction that in reduced form has denominator $a_1^3$, hence we have $0<x_{k-2}<a_1^3$ and $a_1\nmid x_{k-2}$.

By induction, there exists a $t_{k-j}\in \Z_{\geq 1}$ such that $0<x_{k-i}<a_1^{i+1}$ and $a_1\nmid x_{k-i}$ for all $i<j$ and
\[
\lambda_{k-j}= \frac{t_{k-j}a_1^{j-1} -\sum_{i=2}^{s-1}x_{k-j+i}a_{i+1}a_1^{i-2}}{a_1^j}-\frac{a_2x_{k-j+1}}{a_1^{j+1}} \, .
\]
Since $\gcd(a_1,a_2)=1$ and $a_1\nmid x_{k-j+1}$ implies that $\frac{a_2x_{k-j+1}}{a_1^{j+1}}$ fully reduced has denominator $a_1^{j+1}$, it follows that
\[
\lambda_{k-j}=\frac{x_{k-j}}{a_1^{j+1}}
\]
for some $0<x_{k-j}<a_1^{j+1}$ with $a_1\nmid x_{k-j}$ where
\[
x_{k-j}=t_{k-j}a_1^j -\sum_{i=1}^{s-1}x_{k-j+i}a_{i+1}a_1^{i-1} \, .
  \]

From this, it follows that $\lambda_{k-j}=\frac{x_{k-j}}{a_1^{j+1}}>0$ for every $j$.
Hence, for any $j$ such that
\[
  \lambda_{k-j},\lambda_{k-j+1},\ldots,\lambda_{k-j+s-1}\neq 0\, ,
\]
we have
\begin{align*}
  \sum_{i=0}^{s-1}\lambda_{k-j+i}& =\sum_{i=0}^{s-1}\frac{x_{k-j+i}}{a_1^{j+1-i}}\\
                                 & = \frac{\sum_{i=0}^{s-1}x_{k-j+i}a_1^i}{a_1^{j+1}}\\
  & = \frac{a_1^jt_{k-j}-\sum_{i=1}^{s-1}x_{k-j+i}a_1^{i-1}a_{i+1}+\sum_{i=1}^{s-1}x_{k-j+i}a_1^i}{a_1^{j+1}}\\
                                 & = \frac{a_1^jt_{k-j}+\sum_{i=1}^{s-1}x_{k-j+i}(a_1^i-a_1^{i-1}a_{i+1})}{a_1^{j+1}}\\
  & > \frac{1}{a_1}
\end{align*}
where the final inequality follows from $t_{k-j}\geq 1$ and $a_1-a_{i+1}>0$ for $i=1,\ldots,s-1$ by assumption.
Thus, since $p_0=\sum_{j=0}^k\lambda_j$ is the height of $\bp$, we have that
\[
    p_0\geq \left\lfloor \frac{k}{s}\right\rfloor \frac{1}{a_1}
  \]
  as desired.
\end{proof}


\section{Conclusion}\label{sec:conclusion}

It is likely to be difficult to classify all Ehrhart limits in $\Z[[z]]$.
Despite this, there are interesting further questions that arise when considering Ehrhart limits, such as the following.

\begin{question}
  For the simplices in Theorem~\ref{thm:multilimit}, what are the explicit power series that arise as limits in this case?
Do their coefficients eventually satisfy a linear recursion?
\end{question}

\begin{question}
Are there other sequences of simplices $\Delta_d$ indexed by dimension $d$ where the fundamental parallelepiped points exhibit the nested stabilization structure seen in the proof of Theorem~\ref{thm:multilimit}?
\end{question}

\begin{question}
  Can we characterize the Ehrhart limits $f(z)=\sum_{i=0}^\infty f_iz^i\in \Z[[z]]$ satisfying the property that for all sufficiently large $i$, $f_i$ is constant?
  For example, by Theorem~\ref{thm:freesumlimit}, such power series arise as the limit of $Q\oplus S_d$ as $d\to \infty$ when $Q$ contains the origin as an interior point.
\end{question}

\begin{question}
What general methods are there to produce examples of Ehrhart limits that are not polynomials and that do not arise from reflexive polytopes or simplices?
\end{question}

\bibliographystyle{plain}
\bibliography{Braun}

\end{document}